\documentclass[12pt]{amsart}
\usepackage{amsmath,amsfonts,amssymb}

\def\sumstar{\sideset{}{^*}\sum}

\def\le{\leqslant}
\def\leq{\leqslant}
\def\geq{\geqslant}
\def\ge{\geqslant}

\def\frak{\mathfrak}

\renewcommand{\mod}{\mathop{\rm{mod}}}

\numberwithin{equation}{section}

\newtheorem{cor}{COROLLARY}[section]
\newtheorem{prop}[cor]{PROPOSITION}
\newtheorem{lem}[cor]{LEMMA}

\theoremstyle{definition}

\newtheorem{theorem}{THEOREM}

\begin{document}
\title{\bf Note on a note of Goldston and Suriajaya}
\author[Friedlander]{J.B. Friedlander$^*$}
\thanks{$^*$\ Supported in part by NSERC grant A5123}
\author[Iwaniec]{H. Iwaniec}

\maketitle

\medskip

{\bf Abstract:} 
We show that the assumption of a weak form of the Hardy-Littlewood conjecture 
on the Goldbach problem
suffices to disprove the possible existence of exceptional zeros of 
Dirichlet $L$-functions. This strengthens a result of the authors 
named in the title.

\section{\bf Introduction}

A fundamental problem in analytic number theory is that of establishing 
zero-free regions for Dirichlet $L$-functions. In case the corresponding 
character $\chi \mod q$ is complex, or alternatively, in the case of complex 
zeros $\rho= \beta + i\gamma$ with $\gamma \neq 0$, one has long known 
how to produce zero-free regions of the type 
\begin{equation}\label{eq:1.1}
\sigma \geq 1-c/\log q(|t|+2)
\end{equation}  
where $s=\sigma + it$. In the case both $\chi$ and $s$ are 
real much less is known, nothing more recent than a famous ``ineffective'' 
estimate of 
Siegel which replaces $\log q$ by $q^{\varepsilon}$. This exponentially 
weaker result has been a serious impediment to progress in many basic 
questions. 

In the absence of a solution to the problem of these exceptional zeros there 
are some useful results showing their rarity. There have also been
attempts to relate the question to other very difficult problems. 
One famous example, by now folklore, shows that the non-existence of such zeros 
would follow from improvements seductively small in the 
Brun-Titchmarsh theorem which gives uniform upper bounds for the number of 
primes in an arithmetic progression. 

In more recent years some results have been obtained by showing how relatively 
good bounds for exceptional zeros would follow from assumptions about the 
seemingly less directly related Goldbach conjecture.
The latter famous statement predicts that every even integer exceeding two 
can be written as the sum of two primes. Hardy and Littlewood [HL] put 
forth a conjectured asymptotic formula for the number of representations of $n$ 
as the sum of two primes. Following the normal practice in the subject, we 
find it simpler to consider a weighted sum over the repsentations involving 
the von Mangoldt function, one which leads to an entirely equivalent 
conjecture. Let 
\begin{equation}\label{eq:1.2}
G(n)=\sum_{\substack{m_1+m_2=n\\2\nmid m_1m_2}}\Lambda(m_1)\Lambda(m_2) . 
\end{equation}  
The Hardy-Littlewood conjecture predicts that, for $n$ even, we have
$G(n)\sim {\frak S}(n)n$
 where ${\frak S}(n)$ is a certain positive product over the primes, defined 
in~\eqref{eq:6.2} to~\eqref{eq:6.4} and easily large enough to imply Goldbach 
for all large $n$.  

A rather weakened, but still formidable, form of the Hardy-Littlewood 
conjecture which features in this (and in earlier) work is as follows.

\bigskip 

{\bf Weak Hardy-Littlewood-Goldbach Conjecture}: {\sl For all 
sufficiently large even $n$ we have}
\begin{equation}\label{eq:1.3}
\delta {\frak S}(n)n < G(n) < (2-\delta){\frak S}(n)n, 
\end{equation} 
{\sl for some fixed} $0<\delta <1$.

\medskip
\noindent Despite having a slightly different name, this is the same conjecture 
employed in [GS]. 
We are going to prove the following result. 

\begin{theorem} 
Assume that the Weak Hardy-Littlewood-Goldbach Conjecture~\eqref{eq:1.3} 
holds for all sufficiently large even $n$. Then there are no zeros of 
any Dirichlet $L$-function 
in the region~\eqref{eq:1.1} with a positive constant $c$ which is now 
allowed to depend on $\delta$.
\end{theorem}

Of course, as remarked above, we are allowed to restrict the proof to the 
case of real zeros of $L$-functions with real characters. We can also assume 
that the modulus $q$ is large since if there were only a finite number of 
exceptional zeros we could reach a region of type~\eqref{eq:1.1} simply by 
adjusting the constant $c$ (since $L(1,\chi)> 0$).  

Precursors of this work 
include 
[Fe], [BHMS], [BH], [Ji] and [GS]. A discussion of their results 
can be found in [GS]. 
In those works the results led to an interval of width $c/(\log q)^2$ rather 
than $c/\log q$.
Aside from an only slightly different 
approach, our main innovation is the use of 
a theorem of Bombieri, described in Section 4, which allows us to incorporate 
the zero repulsion principle used, for example, in connection with Linnik's 
theorem on the least prime in an arithmetic progression. 

\medskip

{\bf Acknowledgement} We were completely motivated to consider this problem by 
the paper [GS] of Goldston and Suriajaya and by the beautiful lecture of Dan 
Goldston delivered to the web seminar of the American Institute of Mathematics 
on May 4, 2021. We are also grateful to the Institute for having provided 
this venue.

\section {\bf A Goldbach number generating series}
 
We let $N$ and $q$ be given positive integers. We consider the sum 
\begin{equation}\label{eq:2.1}
S(q)=\sum_{n\equiv 0 (\mod q)}G(n)e^{-n/N}
\end{equation}  
where $G(n)$ is given by~\eqref{eq:1.2}.
We are going to estimate $S(q)$ in two different ways and 
then compare the results. 

\medskip

\section {\bf Zeros of $L$-functions and an explicit formula}

By the orthogonality property for the characters we write 
\begin{equation*}
\begin{aligned}
S(q)  & =  \sum_{\substack{m_1 +m_2\equiv 0 (\mod q)\\2\nmid m_1m_2}}\Lambda(m_1)\Lambda(m_2)e^{-(m_1+m_2)/N}  \\
& = \frac{1}{\varphi (q)}  
\sum_{\chi (\mod q)}\chi(-1)\sum_m\Bigl(\chi(m)\Lambda(m)e^{-m/N}\Bigr)^2
+ O\Bigl(N(\log q)^2\Bigr)\ 
\end{aligned}
\end{equation*} 
where the error term accounts for a trivial estimation of the contribution 
to the sum from terms $m_1$, $m_2$ with $(m_1m_2, 2q) \neq 1$. 

\smallskip

Let $N\geqslant q$.
The principal character contributes to $S(q)$ an amount 
\begin{equation}\label{eq:3.1}
S_0(q) 
= \frac {1}{\varphi(q)} \Bigl( \sum_{(m,q)=1}\Lambda(m)e^{-m/N}\Bigr)^2
= \frac{N^2}{\varphi (q)}\Bigl(1+O\bigl(\frac{1}{\log N}\bigr)\Bigr) .
\end{equation}

\smallskip

The exceptional character $\chi_1 (\mod\, q)$ contributes 
\begin{equation}\label{eq:3.2}
S_1(q) = \frac {\chi_1(-1)}{\varphi(q)} P(\chi_1)^2
\end{equation} 
where for any $\chi$ we have  
\begin{equation}\label{eq:3.3}
P(\chi) = \sum_m\chi(m)\Lambda(m) e^{-m/N} . 
\end{equation} 
We have the trivial bound $P(\chi) \ll N$. Hence, all the other characters 
$\chi \neq \chi_0, \chi_1$, contribute at most
\begin{equation}\label{eq:3.4}
S_{\infty}(q) \ll \frac {N}{\varphi(q)} \sum_{\chi \neq \chi_0, \chi_1}|P(\chi)| .
\end{equation} 
We evaluate $P(\chi)$ for each $\chi \neq \chi_0$ as follows: 
\begin{equation}\label{eq:3.5}
\begin{aligned}
P(\chi) & = \frac {1}{2\pi i} \int_{(2)}\Gamma(s)N^s\frac{-L'}{L}(s) ds \\
& = -\sum_{\rho} \Gamma(\rho)N^{\rho} +O\Bigl(N^{\frac 12}(\log N)^2\Bigr) ,
\end{aligned}
\end{equation}  
on moving the line of integration to ${\rm Re \,  s} = \frac 12$. 
Here, $\rho = \beta +i\gamma$ runs over the zeros of $L(s,\chi)$ 
with $\frac 12 \le \beta < 1$,
giving the explicit formula for the sum $P(\chi)$ in~\eqref{eq:3.3}.

\section{\bf  Linnik's zero repulsion \`a la Bombieri}


We have the following result. 

\begin{prop}
Suppose there is a real character $ \chi_1 (\mod q)$ with a real 
zero $\beta_1$ of $L(s, \chi_1)$ in the segment
\begin{equation}\label{eq:4.1}
1-\frac{c}{\log q} \le \beta_1 <1
\end{equation} 
where $c$ is a small positive constant.  Let $N(\alpha,T)$ denote 
the number of zeros, $ \rho =\beta + i\gamma$, $\rho \ne \beta_1$ of 
$L(s, \chi)$ for all $\chi (\mod q)$, counted with multiplicity, in the region 
\begin{equation}\label{eq:4.2}
\alpha \le \beta <1, \,\, |\gamma| \le T . 
\end{equation} 
Then, for $\frac 12 \le \alpha \le 1, T\ge 2$ 
\begin{equation}\label{eq:4.3}
N(\alpha, T)\ll (1-\beta_1 )(\log q) q^{b(1-\alpha)}T^3
\end{equation} 
where $b$ and the implied constant are absolute.
\end{prop} 

This proposition follows from Th\'eor\`eme 14  of Bombieri [B] 
which states the following.  
\begin{prop} For $\frac 12 \le \alpha \le 1$ and  $T\ge 2$, we have
\begin{equation*}
\sum_{q\le T}\, \sumstar_{\chi(\mod q)}N(\alpha, T, \chi)
\ll (1-\beta_1)(\log T)T^{b(1-\alpha)}
\end{equation*}
where $\chi$ runs over primitive 
characters and $N(\alpha, T, \chi)$ is the number of zeros 
$ \rho =\beta + i\gamma$, of $L(s, \chi)$ counted with multiplicity, 
in the region~\eqref{eq:4.2}, apart from one real and simple zero 
$\beta_1 \geqslant 1-c/\log T$. 
\end{prop}

To deduce Proposition 4.1 we first consider the case $T\le q$. 
Now, our sum $N(\alpha, T)$ satisfies 
\begin{equation*}
\begin{aligned}
N(\alpha, T) & \le N(\alpha, q) \le \sum_{\chi (\mod q)} N(\alpha, q, \chi)
= \sum_{k | q}\sumstar_{\chi (\mod k)} N(\alpha, q, \chi) \\
& \le \sum_{k \le q}\,\sumstar_{\chi (\mod k)} N(\alpha, q, \chi)
\ll (1-\beta_1)(\log q)q^{b(1-\alpha)} ,
\end{aligned}
\end{equation*}  
giving the result in this case.

On the other hand, if $T>q$ the result follows from the classical 
bound $N(T, \chi) \ll T\log qT$. 

\medskip

We remark that the factor $T^3$ in~\eqref{eq:4.3} is wasteful but 
simplifies the proof and is unimportant for our application.

\section {\bf  First estimation of $S(q)$}

We are now ready to complete the estimation of $S(q)$ by means of 
Dirichlet characters.
We input the result of Proposition 4.1. We derive by partial summation that 
\begin{equation}\label{eq:5.1}
\sum_{\chi (\mod q)}\, \sum_{\rho \neq \beta_1}|\Gamma(\rho)N^{\rho}|
\ll N(1-\beta_1 )\log q ,
\end{equation} 
provided $N\geq q^{b+1}$. Hence, by~\eqref{eq:3.4} and~\eqref{eq:3.5}, we have
\begin{equation}\label{eq:5.2}
S_{\infty}(q) \ll \frac{N^2}{\varphi (q)}(1-\beta_1 )\log q .
\end{equation} 
On the other hand, taking only the exceptional character rather than 
the others in~\eqref{eq:5.1}, we get    
\begin{equation}\label{eq:5.3}
P(\chi_1) = -\Gamma(\beta_1)N^{\beta_1} + O\bigl(N(1-\beta_1 )\log q\bigr)
\end{equation} 
and so
\begin{equation}\label{eq:5.4}
P(\chi_1)^2 = \Gamma(\beta_1)^2N^{2\beta_1} + O\bigl(N^2(1-\beta_1 )\log q\bigr) .
\end{equation} 

Finally, adding the above estimates we obtain
\begin{equation}\label{eq:5.5}
\begin{aligned}
S(q) & = S_0(q) +\frac{\chi_1(-1)}{\varphi (q)}N^{2\beta_1} 
+ O\Bigl(\frac{N^2}{\varphi(q)}(1-\beta_1 )\log q\Bigr)\\
& = \bigl(1+ \chi_1(-1) + \varepsilon(q,N)\bigr)\frac{N^2}{\varphi(q)} ,
\end{aligned}
\end{equation} 
where 
\begin{equation}\label{eq:5.6}
\varepsilon(q,N)\ll (1-\beta_1 )\log N + (\log N)^{-1}
\end{equation} 
if $N\ge q^{b+1}$. The implied constant is absolute. 

\section {\bf  A model of $S(q)$}

In this section we consider the a model for $S(q)$ to be 
given by~\eqref{eq:2.1}. 
Although our computations in this section, as for those earlier, are 
unconditional, the closeness of the model to the sum $S(q)$ is not. 
That will rest on the assumption in the next section. 

The Hardy-Littlewood conjecture for
the Goldbach problem asserts that, for $n$ even, 

\begin{equation}\label{eq:6.1}
G(n) \sim {\frak S}(n)n ,
\end{equation} 
where 
\begin{equation}\label{eq:6.2}
{\frak S}(n)= 2CH(n) ,
\end{equation} 
and, in turn, $C$ is the positive absolute constant 
\begin{equation}\label{eq:6.3}
C =\prod_{p>2}\Bigl(1-\frac{1}{(p-1)^2}\Bigr) ,
\end{equation} 
and $H(n)$ is the multiplicative function 
\begin{equation}\label{eq:6.4}
H(n) = \prod_{\substack{p|n\\p>2}}\Bigl(1+\frac{1}{(p-2)}\Bigr) . 
\end{equation} 

\medskip 

\begin{lem}
We have 
\begin{equation}\label{eq:6.5}
\sum_{\substack{n\equiv 0(\mod q)\\n\,\,{\rm even}}}{\frak S}(n)ne^{-n/N}
= \frac{N^2}{\varphi (q)}\Bigl(1+O\Bigl(\frac{q}{N}\Bigr)^{\frac12}\Bigr) .
\end{equation} 
\end{lem}

\begin{proof}
Pulling out the constant factor $2C$ we evaluate the sum
\begin{equation}\label{eq:6.6}
\sum_{\substack{n\equiv 0(\mod q)\\n\,\,{\rm even}}}H(n)ne^{-n/N}
= \frac{N}{2\pi i}\int_{(2)}\Gamma(s+1)N^sZ_q(s)ds
\end{equation} 
where 
\begin{equation*}\
Z_q(s)= \sum_{n\equiv 0(\mod k)} H(n)n^{-s}
\end{equation*} 
and $k =[2,q] = 2q/(2,q)$. We have
\begin{equation*}
\begin{aligned}
 Z_q(s) & = H(k) k^{-s}\sum_{\ell =1}^{\infty}\prod_{\substack{p|\ell\\p\nmid k}}
\Bigl(1+\frac{1}{p-2}\Bigr)\ell^{-s}\\
& = H(k)k^{-s}\prod_{p|k}\Bigl(1-\frac{1}{p^s}\Bigr) \prod_{p\nmid k}
\Bigl( 1+ \Bigl( 1+\frac{1}{p-2}\Bigr) \frac{1}{p^s}
\Bigl( 1-\frac{1}{p^s}\Bigr)^{-1}\Bigr)\\
& = H(k)k^{-s}\zeta(s)\prod_{p\nmid k}\Bigl(1+\frac{1}{p^s(p-2)}  \Bigr) .
\end{aligned}
\end{equation*} 
This has a simple pole at $s=1$ with residue 
\begin{equation}\label{eq:6.7}
{\rm res}_{s=1}Z_q(s)= \frac{H(k)}{k}\prod_{p\nmid k}\Bigl(1+\frac{1}{p(p-2)}\Bigr)
=\frac{1}{2C\varphi(q)} .
\end{equation} 
Moving the line of integration in~\eqref{eq:6.6} to 
${\rm Re \,  s} = \frac 12$, we obtain~\eqref{eq:6.5}.

\end{proof}

\section {\bf  Completion} 

In this section we begin by completing our second estimation for $S(q)$ by 
comparing its model. In this we need to draw on an unproved 
assumption. We assume the following. 

\bigskip 

{\bf Weak Hardy-Littlewood-Goldbach Conjecture}: {\sl For all 
sufficiently large even $n$ we have}
\begin{equation}\label{eq:7.1}
\delta {\frak S}(n)n < G(n) < (2-\delta){\frak S}(n)n, 
\end{equation} 
{\sl where} $0<\delta <1$ {\sl is fixed.} 

\bigskip

From~\eqref{eq:7.1} and Lemma 6.1 we see that, with $N\geq q^2$,
\begin{equation}\label{eq:7.2}
\delta\frac{N^2}{\varphi (q)}\Bigl(1+O\Bigl(\frac{1}{\log N}\Bigr) \Bigr) 
 < S(q) < (2-\delta) \frac{N^2}{\varphi (q)}
\Bigl(1+O\Bigl(\frac{1}{\log N}\Bigr)\Bigr)  .
\end{equation}

We now combine our two estimations for $S(q)$. By~\eqref{eq:7.2} 
and~\eqref{eq:5.5} we have
\begin{equation}\label{eq:7.3}
\delta  < 1+ \chi_1(-1) +\varepsilon(q,N) < 2-\delta
\end{equation} 
if $N\geq q^{b+1}$, where (a now slightly different) $\varepsilon(q,N)$
still satisfies~\eqref{eq:5.6}. Taking the constant $c$ in~\eqref{eq:4.1} 
to be sufficiently small, we see from~\eqref{eq:5.6} that
\begin{equation}\label{eq:7.3}
|\varepsilon(q,N)| \leq \tfrac{\delta}{2}. 
\end{equation} 

Since $1 +\chi_1(-1) = 0$ or $= 2$, this is not possible.

\bigskip

{\bf Conclusion}:  The assumption~\eqref{eq:7.1} with some positive $\delta$ 
implies that there is a real zero-free interval $s\geq 1-c(\delta)/\log q$,  
(and hence a zero-free region in the plane 
$\sigma \geq 1-c(\delta)/\log q(|t|+2)$
for all $L(s,\chi)$ to all moduli $q$). In the case that $\chi_1(-1)=1$ only 
the upper bound in~\eqref{eq:7.1} is required and in the other case, 
that $\chi_1(-1)=-1$, only the lower bound is required.


\medskip 
Department of Mathematics, University of Toronto

Toronto, Ontario M5S 2E4, Canada 

\medskip

Department of Mathematics, Rutgers University

Piscataway, NJ 08903, USA

\end{document}